\documentclass{amsart}
\usepackage{amsfonts}
\usepackage{latexsym}
\usepackage{amssymb}
\usepackage{amsmath}
\usepackage{color}
\usepackage{bbm}
\usepackage{tikz}
\usepackage{enumerate}

\newcommand{\R}{\mathbb R}
\newcommand{\N}{\mathbb N}

\newcommand{\E}{\mathbb E}
\newcommand{\Pro}{\mathbb P}
\newcommand{\dif}{\,\mathrm{d}}

\newcommand{\K}{\ensuremath{{\mathbb K}}}

\newcommand\bd{{\rm bd}}
\newcommand{\bM}{\ensuremath{{\mathbf m}}}
\def\dint{\textup{d}}
\newcommand{\SSS}{\ensuremath{{\mathbb S}}}

\newcommand{\B}{\ensuremath{{\mathbb B}}}

\DeclareMathOperator*{\kmax}{k-max}
\DeclareMathOperator{\conv}{conv}
\DeclareMathOperator{\id}{id}

\newtheorem{thm}{Theorem}[section]

\newtheorem{lemma}[thm]{Lemma}

\newtheorem{proposition}[thm]{Proposition}

\newtheorem{rmk}[thm]{Remark}

\allowdisplaybreaks

\begin{document}


\title[Random convex sets between polytopes and zonotopes]{On the geometry of random convex sets between polytopes and zonotopes}

\author[D. Alonso-Guti\'errez]{David Alonso-Guti\'errez}
\address{Departamento de Matem\'aticas, Universidad de Zaragoza, Zaragoza, Spain}
\email{alonsod@unizar.es}

\author[J. Prochno]{Joscha Prochno}
\address{School of Mathematics \& Physical Sciences, University of Hull, Hull, United Kingdom} \email{j.prochno@hull.ac.uk}

\subjclass[2010]{}

\thanks{...}

\date{\today}

\begin{abstract}
In this work we study a class of random convex sets that ``interpolate'' between polytopes and zonotopes. These sets arise from considering a $q^{th}$-moment ($q\geq 1$) of an average of order statistics of $1$-dimensional marginals of a sequence of $N\geq n$ independent random vectors in $\R^n$. We consider the random model of isotropic log-concave distributions as well as the uniform distribution on an $\ell_p^n$-sphere ($1\leq p < \infty$) with respect to the cone probability measure, and study the geometry of these sets in terms of the support function and mean width. We provide asymptotic formulas for the expectation of these geometric functionals which are sharp up to absolute constants. Our model includes and generalizes the standard one for random polytopes.
\end{abstract}

\maketitle


\section{Introduction and main results}

\subsection{General introduction}

A random polytope in $\R^n$ is the convex hull of $N$ points chosen randomly according to a given law. In fact, several other models to define random polytopes exist, but this model is arguably the most natural, best known and most studied one. It was more than 150 years ago that J. J. Sylvester initiated their study when he posed a problem in The Educational Times in 1864 \cite{Sy1864}. In it, he asked for the probability that four points chosen uniformly at random in an indefinite plane have a convex hull which is a four-sided polygon. Within a year it was understood that the question was ill-posed and Sylvester modified the question, asking for the probability that four points chosen independently and uniformly at random from a convex set $K$ in the plane are in convex position. This problem became known as the famous ``four-point problem'' and was the starting point of extensive research (see also \cite{B} and the references therein).

It were A. R\'enyi and R. Sulanke who later, in their seminal papers \cite{RS1}, \cite{RS2}, \cite{RS3}, focused on the asymptotic of the expected volume of a random polytope as the number of points $N$ tends to infinity. Since then and especially in the last decades, random polytopes found increasing interest. This is to a large extent due to their emergence in various branches of mathematics and their broad spectrum of applications. Among others, random polytopes appear in approximation theory \cite{Mue90,Ba07}, random matrix theory \cite{LPRT05} or in other disciplines such as statistics, information theory, signal processing, medical imaging or digital communications (see \cite{DT09} and the references therein), just to mention a few. Because of their ``pathologically'' bad behaviour, they are also a major source for counterexamples, as can be seen, for instance, in \cite{Gl81} or \cite{LRT14}. Some of the important quantities studied in order to understand their geometric structure are expectations, variances, and distributions of functionals associated to the random polytope, for instance, the volume, the number of vertices, intrinsic volumes, mean outer radii and, in particular, the mean width.

Obviously, the behavior of these geometric functionals depends on the underlying model of randomness. There are two such models that have drawn a particularly lot of attention and have been studied extensively. One situation is the case in which the random vectors generating the polytope are Gaussian and results in this direction can be found, for instance, in \cite{Gl81,Sz83,MT03,HMR04,LPRT05,KK09} and the references given therein. The other one is the case when the points that span the polytope are chosen uniformly at random inside a convex body $K$. Here, we may refer the reader to \cite{DGT1,DGT2,ADHP14,AP1} and again the references given there. Typically, $K$ is considered to be isotropic and the geometry of the random polytope relates to the isotropic constant of $K$. These two models are particular situations of the case when the random vectors are distributed according to an isotropic log-concave probability, which is the general framework in which they are studied.

In the work \cite{GLSW2}, extending the previous works \cite{GJ97, GJN97, GJ99}, Y. Gordon, A. E. Litvak, C. Sch\"utt and E. Werner studied the geometry of the unit balls and their polars of the norm given by
$$
\Vert x\Vert_{\ell,q}=\left(\sum_{k=1}^\ell\kmax_{1\leq i\leq N}|\langle x,a_i\rangle|^q\right)^{1/q},\qquad x\in\R^n,
$$
where $1\leq q\leq\infty$, $\{a_i\}_{i=1}^N$ is a fixed sequence of vectors spanning $\R^n$, $1\leq \ell\leq N$, and $\kmax_{1\leq i\leq N}|\langle x,a_i\rangle|$ is the $k^{th}$ largest number in the set $\{|\langle x,a_i\rangle| \}_{i=1}^N$. As different choices of the involved parameters show, this class of convex bodies is quite rich, which also explains the interest in those spaces. To be more precise, when we choose $\ell=1$, then the polar body of this unit ball is the symmetric convex hull of the vectors $a_1,\dots,a_N$. On the other hand, if we let $\ell=N$, then the polar of the unit ball of this norm is just a linear transformation of a projection of the unit ball of $\ell_{q^*}^N$ onto an $n$-dimensional subspace, where $q^*$ is the conjugate of $q$. In particular, choosing $q=1$ and $\ell=N$, the polar body of the unit ball of $\Vert\cdot\Vert_{\ell,q}$ is a zonotope. For $q=1$, the polar of the unit ball is a linear image of a projection of $\big(\ell \B_1^N\big)\cap \B_\infty^N$ (see Lemma 5.1 in \cite{GLSW2}).

Here, we introduce a probabilistic variant of this by considering the vectors $a_1,\dots,a_N\in\R^n$ not to be fixed, but chosen independently at random according to a given probability law on $\R^n$ (details are given below). This is, in fact, quite interesting and natural, because a rich family of random convex sets arises that includes the important class of random polytopes, but extends beyond that classical and well understood setting. In this new model, the definition of the random convex sets takes more order statistics of $1$-dimensional marginals and higher moments into account. It therefore should capture more information about the geometry and distribution of mass. In this work, we study how sensitive this information is in the number of order statistics and moments considered and initiate the study of this new and more general class of random convex sets, restricting ourselves to the expectation of the mean width for now. To be more precise, we will study the geometry of this family of random convex bodies for several models of randomness and the dependence of their geometric parameters on the space dimension $n$, the number $N$ of vectors generating them, the number $\ell$ of order statistics considered and their dependence on the moment $q$. We will compute, up to absolute constants, the expected value of the mean width of the polar bodies of the unit balls of $\Vert\cdot\Vert_{\ell,q}$ when the independent random vectors $a_1,\dots,a_N$ are distributed according to an isotropic log-concave probability law or chosen uniformly at random from an $\ell_p^n$-sphere according to the cone probability measure. The precise statements are given in the following subsection. Our proofs reflect a lively interplay between geometric arguments with techniques and methods from analysis and probability as it is typical in Asymptotic Geometric Analysis. They also underline the role that order statistics play and their interplay with classical elements of Functional Analysis such as Orlicz spaces.

\subsection{Presentation of the main results}

Before we present our main results, we need to fix some notation, which differs from the one used in \cite{GLSW2}. All random objects will be defined on the same probability space $(\Omega,\mathcal A,\Pro)$. Let $\ell,n,N\in\N$ with $N\geq n$ and $1\leq \ell \leq N$. For $1\leq q<\infty$ and random vectors $X_1,\dots,X_N$ in $\R^n$, we define a random convex body $K_{N,\ell,q}$ in $\R^n$ by its support function, which is given by
\begin{align}\label{eq:definition support function}
h_{K_{N,\ell,q}}(\theta) :=\left(\frac{1}{\ell}\sum_{k=1}^\ell \kmax_{1\leq i \leq N} |\langle X_i,\theta\rangle|^q\right)^{1/q}, \qquad \theta\in \SSS^{n-1}.
\end{align}
Note that for any realization of the random vectors the function $h_{K_{N,\ell,q}}:\R^n\to\R$ is positive homogeneous and subadditive and, as such, there exists indeed a unique convex body whose support function is $h_{K_{N,\ell,q}}$. As already introduced above, we call this body $K_{N,\ell,q}$.
The mean width of this random convex set is thus defined as
\[
 w\big(K_{N,\ell,q}\big) := \int_{\SSS^{n-1}} h_{K_{N,\ell,q}}(\theta) \,\dint\sigma_{n-1}(\theta),
\]
where $\SSS^{n-1}$ is the Euclidean unit sphere, which is naturally equipped with a Borel $\sigma$-field, and $\sigma_{n-1}$ is the unique uniform probability measure on it.

In order to shorten notation, we will write $a\approx b$ to denote equivalence up to absolute constants, that is, the existence of constants $c_1,c_2\in(0,\infty)$ that do not depend on any of the parameters involved such that $c_1a\leq b\leq c_2 a$.

We will prove the following asymptotic formula for the mean width for the isotropic log-concave random model.

\begin{thm}\label{TheoremMEanWidthGeneralq}
	Let $n,N\in\N$ with $n\leq N\leq e^{\sqrt n}$ and let $X_1,\dots,X_N$ be independent random vectors in $\R^n$ distributed according to an isotropic log-concave probability law $\mu$ on $\R^n$.
	Then, for all $1\leq \ell\leq N$ and any $q\geq 1$,
	\[
	\E\, w\big(K_{N,\ell,q}\big)\approx \min\left\{\max\left\{\sqrt q, \sqrt{\log(N/\ell)}\,\right\},\sqrt{\log N}\right\}.
	\]
\end{thm}

Note that this theorem includes one of the main results in \cite[Theorem 1.1]{AP1} by simply choosing $q=1$ as well as $\ell=1$.

\begin{rmk}
The case in which $X_1,\dots,X_N$ are independent standard Gaussian random vectors will be a special case, since the theorem will be proved separately in that case and the proof of the general case will rely on it. In the Gaussian case the restriction $N\leq e^{\sqrt n}$ on the number of random vectors is not needed for the result to hold. This restriction appears in the general isotropic log-concave case since our proof relies on Paouris' tail estimates for the Euclidean norm. When $e^{\sqrt n}\leq N\leq e^n$, the result cannot be true in full generality. Consider, for example, the uniform measure on a dilation of the $\ell_1^n$-ball and take $q=n$ and $N=e^n$ (see Remark \ref{RemarkIsotroic}). However, using recent results proved in \cite{GHT}, some estimates can be given.
\end{rmk}

\begin{rmk}
Because of rotationally invariance, in the Gaussian case the result holds not only for the expected value of the mean width but for the expected value of the support function in any particular direction $\theta\in \SSS^{n-1}$.
\end{rmk}

When the random model is given by a uniform distribution on an $\ell_p^n$-sphere with respect to the cone measure, then we obtain the following asymptotic formula. Recall that for a convex body $K$ in $\R^n$ the cone probability measure $\bM_K$ on $\bd\, K$ is defined for measurable $A\subset\bd\, K$ to be the (Lebesgue) volume of the cone with base $A$ and cusp $0$, normalized by the volume of $K$.

\begin{thm}\label{TheoremBpnq}
	Let $1\leq p < \infty$ and $n,N\in\N$ with $n\leq N\leq e^{\sqrt{n}}$, $c\in(0,\infty)$ being an absolute constant. Let $X_1,\dots X_N$ be independent random vectors distributed on the $\ell_p^n$-sphere according to the cone probability measure $\bM_{\B_p^n}$.
	Then, for every $1\leq \ell\leq N$ and any $q\geq 1$,
	\[
	\E\, w\big(K_{N,\ell,q}\big) \approx n^{-\frac{1}{p}}\min\left\{\max\left\{\sqrt q, \sqrt{\log(N/\ell)}\,\right\},\sqrt{\log N}\right\}.
	\]
\end{thm}

\begin{rmk}
Our proof shows that for $1\leq p<2$ the upper estimate in the previous theorem holds when $n\leq N\leq e^{cn^{p/2}}$ and it holds for $N\geq n$ whenever $p\geq 2$.
\end{rmk}

After having presented our main results, let us comment a little bit on the main ideas in their proofs. We refer to Section \ref{sec:prelim} below for any unexplained notion or notation. A key ingredient in both of them is the following: if we are given a sequence $\xi_1,\dots,\xi_n$ of independent, identically distributed and integrable random variables and define for each $1\leq \ell \leq n$ an Orlicz function by
\[
M_\ell(s) = \int_0^s \int_{|\xi_1|\geq 1/(t\ell)} |\xi_1|\,\dif \Pro \dif t,
\]
then, as was shown in \cite{GLSW1} (see also \cite{LPP16} for extensions),
\[
\E \,\frac{1}{\ell}\sum_{k=1}^\ell \kmax_{1\leq i \leq n} |x_i\xi_i| \approx\frac{1}{\ell}\, \|x\|_{M_\ell},\qquad x=(x_1,\dots,x_n)\in\R^n.
\]
In view of Definition \eqref{eq:definition support function} of the support function of $K_{N,\ell,q}$, to estimate its expectation, we apply the previously mentioned result to the sequence of $1$-dimensional marginals $|\langle X_1,\theta \rangle|^q,\dots,|\langle X_N,\theta \rangle|^q$, $\theta\in\SSS^{n-1}$ and the choice $x=(1,\dots,1)$. Here, the distribution of the random vectors $X_1,\dots,X_N\in\R^n$ is in each case given by the underlying model of randomness, so it is either isotropic log-concave on $\R^n$ or uniformly distributed on an $\ell_p^n$-sphere with respect to the cone probability measure. Roughly speaking it is then left to compute the Orlicz function and the corresponding Orlicz norm of the vector $(1,\dots,1)\in\R^N$, which in each case requires obtaining sharp (up to absolute constant) lower and upper bounds. This is where various tools and ideas of geometric and probabilistic flavor enter in the proofs, among others
\renewcommand\labelitemi{\tiny$\bullet$}
\begin{itemize}
\item the famous theorem of Paouris on the deviation of the Euclidean norm on an isotropic convex body (Proposition \ref{thm:paouris}),
\item the geometry and relation of floating and $L_q$-centroid bodies shown by Paouris and Werner (see Remark \ref{rem:relation floating and Lq centroid bodies}),
\item the probabilistic representation and the concentration of the cone probability measure on $\ell_p^n$-spheres due to Schechtman and Zinn (Theorem \ref{thm:deviation cone measure}),
\end{itemize}
where the first two are used in the context of isotropic log-concave measures.

\subsection{Organization of the paper}

The rest of the paper is organized as follows. In Section \ref{sec:prelim} we provide the necessary preliminaries, which are organized and presented by topic. Section \ref{sec:proofs} contains the proofs of the three main theorems, where we devote to each random model its own subsection. Subsection \ref{subsec:general estiamtes} provides some general results for random vectors in $\R^n$, Subsection \ref{subsec:gaussian} covers the Gaussian random model, which will be proved separately, Subsection \ref{subsec:isotropic} contains the general isotropic log-concave case when the number $N$ of points satisfies $n\leq N \leq e^{\sqrt{n}}$, Subsection \ref{subsec:isotropicManyPoints} contains some estimates that hold when the number $N$ of points exceeds $e^{\sqrt{n}}$ and, in Subsection \ref{subsec:spheres}, distributions with respect to the cone probability measure on an $\ell_p^n$-sphere are considered.

\section{Preliminaries}\label{sec:prelim}

Before we proceed with the proofs of our main results, we introduce all the necessary background material needed throughout this paper. We subdivide those preliminaries into various topics.

\subsection{General background and notation}

The natural number $n\in\N$ always denotes the dimension of the space. We write $\langle\cdot,\cdot\rangle$ for the standard inner product on $\R^n$. As usual, let $\B_2^n=\{x\in\R^n\,:\, \|x\|_2\leq 1 \}$ and $\SSS^{n-1} = \{ x \in \R^n : \|x\|_2 = 1\}$ be the unit ball and unit sphere in the Euclidean space $\R^n$, respectively. We write $\sigma_{n-1}$ for the uniform probability measure on $\SSS^{n-1}$, which is the unique rotationally invariant Haar probability measure, and equip $\SSS^{n-1}$ with its natural Borel $\sigma$-field.

Let $(\Omega,\mathcal A,\Pro)$ and $(F,\mathcal F,\mu)$ be two probability spaces. For a random variable $X:\Omega\to F$, we write $X\sim \mu$ if and only if the law of $X$ is $\mu$, that is, $\Pro \circ X^{-1}=\mu$.

For two sequences $(a(n))_{n\in\N}$ and $(b(n))_{n\in\N}$ of real numbers, we write $a(n)\gtrsim b(n)$ (or $a(n)\lesssim b(n)$) provided that there is a constant $c\in(0,\infty)$ such that $a(n)\geq cb(n)$ (or $a(n)\leq cb(n)$) for all $n\in\N$. Moreover, we write $a(n)\approx b(n)$ if $a(n) \lesssim b(n)$ and $a(n) \gtrsim b(n)$.

\subsection{ Convex bodies and isotropic log-concave probability measures}
A convex body $K\subset \R^n$ is a compact and convex set with non-empty interior and we denote by $\K^n$ the set of all convex bodies in $\R^n$. A convex body $K\in\K^n$ is called symmetric if $-x\in K$, whenever $x\in K$.
We will denote its volume (or Lebesgue measure) by $|\cdot |$, the dimension being understood from the context.

A convex body $K\in\K^n$ is said to be isotropic if $|K|=1$, it has center of mass at the origin and satisfies the isotropic condition
\[
\int_K\langle x,\theta\rangle^2 \, \dint x=L_K^2,\qquad \theta\in \SSS^{n-1},
\]
where $L_K$ is a constant independent of $\theta\in\SSS^{n-1}$, which is called the isotropic constant of $K$.

Let $K\in\K^n$. The support function $h_K:\R^n\to\R$ of $K$ is defined by
\[
h_K(y):=\max_{x\in K}\,\langle x,y \rangle
\]
and the mean width of $K\in\K^n$ is
\[
w(K):=\int_{\SSS^{n-1}}h_K(\theta) \, \dint\sigma_{n-1}(\theta).
\]

The {cone probability measure} $\bM_K$ of $K\in\K^n$ is defined as
\[
\bM_K(B) := \frac{\big|\{rx:x\in B\,,0\leq r\leq 1\}\big|}{|K|}\,,
\]
where $B\subset\bd\,K$ is a Borel subset of the boundary $\bd\,K$ of $K$. For $K\in\K^n$ one has that
\begin{align*}
\int_{\R^n}f(x)\,\dint x = n\,|K|\int_0^\infty\int_{\bd\,K}f(ry)\,r^{n-1}\,\dint\bM_K(y)\,\dint r
\end{align*}
for all non-negative measurable functions $f:\R^n\to\R$, which, in fact, may alternatively be used as a definition for the cone measure $\bM_K$ of $K$ (see, for instance, \cite[Proposition 1]{NR03}).
Let us remark that the cone measure of the unit ball of $\ell_p^n$ coincides with the normalized surface area measure if and only if $p\in\{1,2,\infty\}$.

A Borel probability measure $\mu$ on $\R^n$ is called log-concave if for all non-empty compact sets $A,B\subset\R^n$ and all $0<\lambda<1$, we have
\[
\mu\big( (1-\lambda)A+\lambda B\big)\geq \mu(A)^{1-\lambda}\mu(B)^\lambda\,.
\]
Similarly, a function $f:\R^n\to[0,\infty)$ is said to be log-concave if for all $x,y\in\R^n$ and any $0<\lambda<1$,
\[
f\big((1-\lambda)x+\lambda y\big) \geq f(x)^{1-\lambda}f(y)^{\lambda}.
\]
Let us denote by $\mathcal P_n$ the class of Borel probability measures on $\R^n$ which are absolutely continuous with respect to the Lebesgue measure. By a result of C. Borell (see \cite{Bor75}), we know that every log-concave probability $\mu$ on $\R^n$, which is not fully supported on any hyperplane, belongs to the class $\mathcal P_n$ and has a log-concave density $f_\mu$. Conversely, by the Pr\'ekopa-Leindler inequality, every measure with a log-concave density is log-concave. Typical examples of log-concave measures are the uniform measure on a compact, convex set $K\subset\R^n$ with non-empty interior and volume $1$ or the standard Gaussian measure on $\R^n$.

We say that $\mu\in\mathcal P_n$ is centered if its barycenter is at the origin, i.e., if for every $\theta\in\SSS^{n-1}$,
\[
\int_{\R^n} \langle x,\theta \rangle \dif\mu(x) = 0.
\]
A probability measure $\mu\in\mathcal P_n$ is said to be isotropic if it is centered and satisfies the isotropic condition,
\[
\int_{\R^n}\langle x,\theta \rangle^2 \dif\mu(x)=1
\]
for any $\theta\in\SSS^{n-1}$.

Given a probability space $(\Omega,\mathcal A,\Pro)$, a random vector $X:\Omega\to\R^n$ will be called log-concave if its distribution $\mu:=\Pro^X$ is a log-concave probability measure on $\R^n$. We say that $X:\Omega\to\R^n$ is isotropic if $\mu$ is isotropic, which is then equivalent to $\E[X]=0$ and $\E[X\otimes X]=\id_{\R^n}$. Note that the standard Gaussian measure on $\R^n$ or the uniform probability measure on $L_K^{-1}K$, where $K\in\K^n$ is isotropic, are examples of isotropic log-concave probability measures. For more detailed information, we refer the reader to the monographs \cite{IsotropicConvexBodies,AsymGeomAnalyBook}.


\subsection{Floating bodies and $L_q$-centroid bodies}

Let $\delta>0$ and consider a log-concave probability measure $\mu$ on $\R^n$. We define the (convex) floating body $K_\delta$ of $\mu$ to be
\[
K_\delta=\bigcap_{\theta\in S^{n-1}}\left\{x\in\R^n\,:\,|\langle x,\theta\rangle|\leq t_\theta\right\},
\]
where $t_\theta= \sup\big\{t>0\,:\,\mu(\{ x\in\R^n\,:\, |\langle x,\theta\rangle|\leq t\})=1-\delta\big\}$. 
In the case that $\mu$ is the uniform probability on a convex body $K$, the definition of the floating body goes far back to W. Blaschke \cite{Bl1923} in dimensions 2 and 3, and K. Leichtweiss \cite{L1986} for general space dimensions $n$. The convex floating body, which is defined above, was introduced by C. Sch\"utt and E. Werner in \cite{SW1990}.

The $L_q$-centroid body $Z_q(\mu)$ is the unique convex body with support function
\[
h_{Z_q(\mu)}(y):= \bigg(\int_{\R^n} |\langle x,y \rangle |^q \dif\mu(x)\bigg)^{1/q},\qquad y\in\R^n.
\]
We remark that, using the language of centroid bodies, the condition that a log-concave probability measure $\mu$ is isotropic can be rephrased by saying that $Z_2(\mu)$ is a Euclidean ball. These bodies were originally introduced by E. Lutwak and G. Zhang in \cite{LuZh} under a different normalization. Their study, considered from an asymptotic point of view, was initiated by G. Paouris in \cite{Paouris2,Paouris, Paouris3}. His results were originally written in the context of convex bodies. Nevertheless they also hold in the more general setting of log-concave measures. He studied the mean width of the $L_q$-centroid bodies of an isotropic log-concave probability $\mu$ and proved that if $1\leq q \leq \sqrt{n}$,
\[
w(Z_q(\mu)) \simeq \sqrt{q}.
\]
Moreover, in \cite[Theorem 1.1]{Paouris3}, he obtained the following famous result on the tail behavior of the Euclidean norm of an isotropic log-concave random vector.

\begin{proposition}[Paouris' theorem]\label{thm:paouris}
There exists an absolute constant $c\in(0,\infty)$ such that for every isotropic log-concave probability measure $\mu$ on $\R^n$ and any $t\geq 1$,
\[
\mu\big(\{x\in\R^n\,:\, \|x\|_2 \geq ct\sqrt{n} \}\big) \leq e^{-t\sqrt{n}}.
\]
\end{proposition}

\begin{rmk}
As a consequence of Paouris' theorem, if $\mu$ is an isotropic log-concave probability measure on $\R^n$ and $X_1,\dots,X_N$ ($N\in\N$) are independent, identically distributed random vectors with probability law $\mu$, then
\[
\E \max_{1\leq i \leq N} \|X_i\|_2\leq(c+1) \max\{\sqrt{n},\log N\},
\]
where $c\in(0,\infty)$ is the constant from Proposition \ref{thm:paouris}.
\end{rmk}

Estimates for the mean width of the centroid bodies of an isotropic log-concave measure when $q\in[\sqrt{n},n]$ have recently been given in \cite{Mi2015}.

$L_q$-centroid bodies are intimately related to the geometry of random convex sets. This relation can be seen, for instance, in \cite{DGT1}, where one of the main results shows that if $X_1,\ldots,X_N$ are independent random points that are selected according to a log-concave measure $\mu$, and  $K_N$ is the random polytope
$K_N=\conv\big(\{\pm X_1,\ldots,\pm X_N\}\big)$,
then
\begin{equation*} 
K_N \supseteq c_1 Z_{\log(2N/n)}(\mu)
\end{equation*}
with probability at least $1-e^{-c_2\sqrt{N}}$, where $c_1,c_2\in(0,\infty)$ are absolute constants.

The following lemma reflects the close relationship between the centroid bodies and the floating bodies. It is simply the isotropic log-concave analogue to \cite[Theorem 2.2]{PW}.

\begin{lemma}\label{rem:relation floating and Lq centroid bodies}
There exists $c_1,c_2\in(0,\infty)$ such that, for all $n\in\N$ and every isotropic log-concave measure $\mu$ on $\R^n$ and any $\delta\in(0,\frac{1}{e})$,
\[
c_1 Z_{\log(\frac{1}{\delta})}(\mu) \subseteq K_\delta \subseteq c_2 Z_{\log(\frac{1}{\delta})}(\mu).
\]
\end{lemma}
\begin{proof}
The proof follows directly along the lines of the proof given in \cite[Theorem 2.2]{PW} using the isotropic log-concave analogues of the ingredients used there.
\end{proof}

\subsection{Geometry of $\ell_p^n$-balls}

For $n\in\N$ and  $1 \leq p \leq \infty$, we denote by $\ell_p^n$ the space $\R^n$ equipped with the norm
\[
\|(x_1,\dots,x_n)\|_p:=
\begin{cases}
\big(\sum_{i=1}^n |x_i|^p\big)^{1/p}\,, & 1 \leq p <\infty,\\
\max_{1\leq i \leq n}|x_i|\,,& p=\infty.
\end{cases}
\]
We write $\B_p^n:=\{x\in\R^n\,:\, \|x\|_p\leq 1 \}$ for the unit ball of $\ell_p^n$ and we let $\SSS_p^{n-1}:=\{x\in\R^n\,:\, \|x\|_p=1 \}$ be the unit sphere in $\ell_p^n$. It is convenient for us to write $\SSS^{n-1}$ instead of $\SSS_2^{n-1}$. The volume of $\B_p^n$ is given by
\[
|\B_p^n| = \frac{\big(2\Gamma(1+\frac{1}{p})\big)^n}{\Gamma(1+\frac{n}{p})}\,,
\]
see \cite[page 180]{AsymGeomAnalyBook}. It follows directly from Stirling's formula that asymptotically, as $n\to\infty$, $|\B_p^n|^{1/n} \approx n^{-1/p}$.

For independent $g_1,\dots,g_n\sim \mathcal N(0,1)$, the Gaussian random vector $G=(g_1,\dots,g_n)$ in $\R^n$ satisfies (see, for instance, \cite[Lemma 2]{SZ})
\begin{align}\label{eq: p norm of gaussian random vector}
\E\,\Vert G\Vert_p\approx\begin{cases}
n^\frac{1}{p}\sqrt p \,,& p\leq\log n,\cr
\sqrt{\log n}\,,&  p\geq\log n.
\end{cases}
\end{align}
Integration in polar coordinates yields
\[
\E\,\Vert G\Vert_p= w(\B_{p^*}^n)\cdot  \E\,\|G\|_2 \approx \sqrt{n} w(\B_{p^*}^n)
\]
and, therefore, the following estimate for the mean width of $\B_p^n$,
\[
w(\B_p^n)\approx
\begin{cases}
n^{\frac{1}{p}-\frac{1}{2}}\sqrt{p}\,,& 1\leq p^*\leq\log n,\cr
n^{-\frac{1}{2}}\sqrt{\log n}\,,& p^*>\log n,
\end{cases}
\]
where $p^*$ is the conjugate of $p$, defined via the relation $\frac{1}{p}+\frac{1}{p^*}=1$.
\smallskip

We rephrase the following result by G. Schechtman and J. Zinn \cite[Lemma 1]{SZ} (independently obtained by S. T. Rachev and L. R\"uschendorf in \cite{RR91}) that provides a probabilistic representation of the cone measure $\bM_{\B_p^n}$ on $\SSS_p^{n-1}$ (see also \cite{BGMN} for an extension).

\begin{proposition}\label{thm:SZ}
	Let $n\in\N$, $1\leq p < \infty$ and $g_1,\dots,g_n$ be independent random variables distributed according to the density
	\[
	f(t) = \frac{e^{-|t|^p}}{2\Gamma\big(1+{1/p}\big)}, \qquad t\in\R\,.
	\]
	Consider the random vector $G=(g_1,\dots,g_n)\in\R^n$ and put $Y:=G/\|G\|_p$. Then $Y$ is independent of $\|G\|_p$ and has distribution $\bM_{\B^n_p}$.
\end{proposition}

The following result is also due to Schechtman and Zinn \cite[Theorem 3]{SZ}. We will use it with the special choice $q=2$ to treat the case of the sphere in $\ell_p^n$ when $1\leq p <2$.  Roughly speaking, it guarantees that with high probability the norm of a vertex of our random convex set is not too big. We reformulate and use it here in the form of Theorem 2 in \cite{N}, where also a short proof is presented (note that in the statement of the result in \cite{N} a minus sign is erroneously missing).

\begin{proposition}\label{thm:deviation cone measure}
	For every $1\leq p\leq q<\infty$ there exist constants $c=c(p,q)\in(0,\infty)$ and $T=T(p,q)$ only depending on $p$ and $q$ such that, for every $t>T$,
	\[
	\bM_{\B_p^n}\bigg( \| x\|_q \geq \frac{t}{n^{1/p-1/q}}\bigg) \leq \exp\bigg(-\frac{t^p\, n^{p/q}}{c}\bigg)\,.
	\]
	Moreover, if $q=2$ and $2>\gamma p$ for some $\gamma\geq 1$, one can choose both constants $c$ and $T$ independently of $p$.
\end{proposition}

\subsection{Orlicz functions and Orlicz spaces}

A convex function $M:[0,\infty)\to[0,\infty)$ that satisfies $M(0)=0$ and $M(t)>0$ for $t>0$ is called an Orlicz function. The conjugate function of an Orlicz function $M$, which we denote by $M^*$, is given by the
Legendre transform
\[
M^*(x) = \sup_{t\in[0,\infty)}\big[xt-M(t)\big].
\]
For instance, taking $M(t)=p^{-1}t^p$, $p\geq 1$, the conjugate function is given by $M^*(t)=p^{*-1}t^{p^*}$ with $1/p+1/p^*=1$.
The $n$-dimensional Orlicz space $\ell_M^n$ is $\R^n$ supplied with the Luxemburg norm
\[
\Vert{x}\Vert_M = \inf \left\{ \rho>0 \,:\, \sum_{i=1}^n M\left(\frac{|x_i|}{\rho}\right) \leq 1 \right\}.
\]
Note that if $M(t)=t^p$, $1\leq p<\infty$, then we have $\Vert\cdot\Vert_M=\Vert\cdot\Vert_p$.
An Orlicz function $M$ is said to be an $N$-function if
\[
\lim_{t\to0}\frac{M(t)}{t}=0\quad\text{and}\quad\lim_{t\to\infty}\frac{M(t)}{t}=\infty.
\]
This condition ensures that $M^*$ is again an Orlicz function.

The following result was first proved in \cite{GLSW1}. We state and use it in the form obtained in \cite[Theorem 3]{LPP16}.

\begin{thm}\label{thm:main sequences of random variables}
	Let $X_1,\dots,X_n$ be a sequence of independent and identically distributed random variables with $\E |X_1|<\infty$. Let $1\leq \ell \leq n$ and $M_\ell$ be the N-function given by
	\begin{equation} \label{eq:definition M star}
		M_\ell^*\bigg(\int_0^\beta X^*(z) \,\dint z \bigg) = \frac{\beta}{\ell},\qquad 0\leq \beta \leq 1.
	\end{equation}
	Then, for all $x\in\R^n$,
	\[
	c \|x\|_{M_\ell} \leq \E \sum_{k=1}^{\ell} \kmax_{1\leq i \leq n} |x_iX_i| \leq C \|x\|_{M_\ell},
	\]
	where $c,C\in(0,\infty)$ are absolute constants.
\end{thm}

The next remark is essentially taken from \cite{LPP16} (see discussion after Theorem 3 there).

\begin{rmk}\label{rmk:M_l integral and M_l expressed M_1}
	Let $M_\ell^*$ be given as in \eqref{eq:definition M star}. Then, for all $s\geq 0$,
	\begin{equation} \label{eq: eqivalent form M}
	M_\ell(s) = \int_0^s \int_{|X| \geq 1/(t\ell)} |X| \,\dint \Pro\, \dint t.
	\end{equation}
	For $\ell=1$, this was shown in \cite[pp. 4-5]{GLSW5}. A simple computation shows that it holds for
	general $\ell$ as well. Note that, for any $1\leq \ell\leq n$ and every $s>0$, we have $M_\ell(s)=\frac{1}{\ell}M_1(\ell s)$.
\end{rmk}

\section{Proofs of the main results}\label{sec:proofs}

In this section we will present the proofs of our main results. We subdivide this section into several subsections, each covering a certain random model. Before we proceed, let us outline our setting, fix some general notation and make some general remarks.

Let $N,n,\ell\in\N$ so that $N\geq n$ and $1\leq \ell \leq N$. We consider independent random vectors $X_1,\dots,X_N$ in $\R^n$ defined on some probability space $(\Omega,\mathcal A, \Pro)$. Let us recall that, for any $1\leq q<\infty$, we are interested in the (unique) random convex body $K_{N,\ell,q}$ in $\R^n$ that has support function
\[
h_{K_{N,\ell,q}}(\theta) =\left(\frac{1}{\ell}\sum_{k=1}^\ell \kmax_{1\leq i \leq N} |\langle X_i,\theta\rangle|^q\right)^{1/q}, \qquad \theta\in \SSS^{n-1}.
\]
where for some $\omega\in\Omega$, $\kmax_{1\leq i \leq N}|\langle X_i(\omega),\theta \rangle|$ is the $k^{th}$ largest element in the set $\big\{ |\langle X_1(\omega),\theta \rangle|,\dots, |\langle X_N(\omega),\theta \rangle| \big\}$. Notice that $K_{N,1,1}$ corresponds to the standard model for random polytopes, which means that $K_{N,1,1}=\textrm{conv}\{\pm X_1,\dots,\pm X_N\}$. Later, to avoid repetition and to shorten the statements of our results, we will simply write $h_{K_{N,\ell,q}}$ or $K_{N,\ell,q}$ and the underlying random model will be always clear from the context.

Let us continue with three very general and quite simple observations that are going to be used throughout this text.

\begin{lemma}\label{rmk:simple observations} In the setting introduced above, the following hold:
	\vskip 1mm
	\noindent(i) For any fixed $1\leq \ell\leq N$, $h_{K_{N,\ell,q}}(\theta)$ is increasing in $q$.
	\vskip 1mm
	\noindent(ii) For any fixed $q\geq 1$, $h_{K_{N,\ell,q}}(\theta)$ is decreasing in $\ell$.
	\vskip 1mm
	\noindent(iii) Whenever $q\geq \log(\ell)$, we have, for all $\theta\in\SSS^{n-1}$ and all realizations (in $\omega\in\Omega$),
	\[
      e^{-1}h_{K_{N,1,1}}(\theta)	\leq  h_{K_{N,\ell,q}}(\theta) \leq h_{K_{N,1,1}}(\theta).
	\]
	In particular, for every $\theta\in \SSS^{n-1}$,
    \begin{align}\label{eq: equivalence ave support functions}
    \E\,h_{K_{N,\ell,q}}(\theta) \approx \E\, h_{K_{N,1,1}}(\theta).
    \end{align}
    Thus, whenever $q\geq\log(\ell)$, the random convex sets $K_{N,1,1}$ and $K_{N,\ell,q}$ are comparable on average.
\end{lemma}
\begin{proof}
Parts (i) and (ii) are clear. Part (iii) follows from the obvious inequality
$$
\ell^{-1/q}\max_{1\leq i \leq N} |\langle X_i,\theta\rangle|\leq h_{K_{N,\ell,q}}(\theta)\leq \max_{1\leq i \leq N} |\langle X_i,\theta\rangle|,\qquad \theta\in\SSS^{n-1}.
$$
\end{proof}

\subsection{General results for random vectors in $\R^n$}\label{subsec:general estiamtes}

We start with some results for random convex bodies arising from independent random vectors in $\R^n$.

\begin{thm}
Let $n,N\in\N$ with $N\geq n$ and let $X_1,\dots, X_N$ be independent random vectors in $\R^n$.
Then, for all $1\leq \ell\leq N$ and all $\theta\in \SSS^{n-1}$,
\[
c \,\E\,h_{K_{\lfloor N/\ell\rfloor,1,1}}(\theta)\leq \E\, h_{K_{N,\ell,1}}(\theta)\leq C\, \E\,h_{K_{\lceil N/\ell\rceil,1,1}}(\theta),
\]
where $c,C\in(0,\infty)$ are absolute constants.
\end{thm}

\begin{rmk}
Note that depending on the relation between $N$ and $\ell$, the convex bodies $K_{\lfloor N/\ell\rfloor,1,1}$ and $K_{\lceil N/\ell\rceil,1,1}$ might be degenerate. Nevertheless, their support function can be defined for any vector $y\in\R^n$ and their mean width will be understood as the average of the support function on the sphere $\SSS^{n-1}$ and not on a lower-dimensional sphere.
\end{rmk}

\begin{proof}
Let $1\leq \ell \leq N$ and for any $\theta\in \SSS^{n-1}$, let $s_\theta=s_\theta(\ell)\in[0,\infty)$ be chosen in such a way that $M_{\ell}\left((\ell s_\theta)^{-1}\right)=1/N$. Then, by Theorem \ref{thm:main sequences of random variables}, the definition of an Orlicz norm and the choice of $s_\theta$, we obtain
\[
\E\, h_{K_{N,\ell,1}}(\theta)  = \frac{1}{\ell} \,\E\,\sum_{k=1}^{\ell} \kmax_{1\leq i \leq N}|\langle X_i,\theta \rangle|
 \approx \frac{1}{\ell}\,\| (1)_{i=1}^N\|_{M_{\ell}}
 = \frac{1}{\ell M_{\ell}^{-1}(1/N)}
= s_\theta.
\]
On the other hand, by the second part of Remark \ref{rmk:M_l integral and M_l expressed M_1},
\[
\frac{1}{N}=M_\ell\left(\frac{1}{\ell s_\theta}\right)=\frac{1}{\ell}\,M_1\left(\frac{1}{s_\theta}\right).
\]
This means that $M_1\left(1/s_\theta\right)=\ell/N$ and therefore,
\[
\sum_{i=1}^{\lfloor N/\ell\rfloor}M_1\left(\frac{1}{s_\theta}\right)=\lfloor N/\ell\rfloor\frac{\ell}{N}\leq 1.
\]
Consequently, the Orlicz norm defined by $M_1$ on the space $\R^{\lfloor N/\ell\rfloor}$ is bounded above by $s_\theta$ for the vector $(1,\dots,1)\in \R^{\lfloor N/\ell\rfloor}$. Thus, using Theorem \ref{thm:main sequences of random variables} with the choices $\ell=1$ and $n=\lfloor N/\ell\rfloor$ there, there exists an absolute constant $c\in(0,\infty)$ such that
\[
s_\theta \geq \big\| (1)_{i=1}^{\lfloor N/\ell\rfloor}\big\|_{M_1} \geq c\, \E\max_{1\leq i \leq \lfloor N/\ell\rfloor} |\langle X_i,\theta \rangle| = c\,\E\, h_{K_{\lfloor N/\ell\rfloor,1,1}}(\theta),
\]
where $K_{\lfloor N/\ell\rfloor,1,1}=\conv\left\{\pm X_1,\dots,\pm X_{\lfloor N/\ell\rfloor} \right\}$.

In the same way, the Orlicz norm of the vector $(1,\dots,1)\in \R^{\lceil N/\ell\rceil}$ defined by $M_1$ is bounded below by $s_\theta$, which, similarly to the previous argument, shows that
\[
s_\theta\leq C\,\E\, h_{K_{\lceil N/\ell\rceil,1,1}}(\theta),
\]
where $C\in(0,\infty)$ is an absolute constant.
\end{proof}

The previous theorem shows that
the random convex sets $K_{N,\ell,1}$ and $K_{\lceil N/\ell\rceil,1,1}$ in $\R^n$ are comparable on average up to absolute constants.

The next theorem shows a similar estimate when the function defining the body is given by the $q^{th}$ moment of the average of order statistics of $1$-dimensional marginals.

\begin{thm}\label{EstimateGeneralq}
Let $n,N\in\N$ with $N\geq n$, $q\geq 1$ and let $X_1,\dots, X_N$ be independent random vectors in $\R^n$.
Then, for all $1\leq \ell\leq N$ and all $\theta\in \SSS^{n-1}$,
\begin{align*}
c\,\E\max_{1\leq i\leq\lfloor N/\ell\rfloor}|\langle X_i,\theta\rangle| & \leq \E\, h_{K_{N,\ell,1}}(\theta) \\
& \leq \E\, h_{K_{N,\ell,q}}(\theta) \leq C\left(\E\max_{1\leq i\leq\lceil N/\ell\rceil}|\langle X_i,\theta\rangle|^q\right)^{1/q},
\end{align*}
where $c,C\in(0,\infty)$ are absolute constant.
\end{thm}
\begin{proof}
The first inequality is the previous theorem. The second inequality is trivial since, by Lemma \ref{rmk:simple observations}, $\E\, h_{K_{N,\ell,q}}(\theta)$ is increasing in $q$ for any fixed $\theta\in\SSS^{n-1}$. The last inequality is a consequence of Jensen's inequality and the same estimate as in the previous theorem applied to the random variables $|\langle X_i,\theta\rangle|^q$, $1\leq i\leq N$.
\end{proof}

\subsection{Gaussian random vectors}\label{subsec:gaussian}

In this section, we consider random convex sets that arise from $q^{th}$-moments of averages of order statistics of the $1$-dimensional marginals of Gaussian random vectors in $\R^n$.

\begin{lemma}\label{lem: support gaussian gives mean width}
Let $n,N\in\N$ with $N\geq n$ and let $X_1,\dots,X_N$ be independent Gaussian random vectors in $\R^n$. For any $1\leq \ell\leq N$ and $q\geq 1$, let $K_{\ell,q}\subseteq\R^N$ be the (non-random) convex body defined by
\begin{align*}
h_{K_{\ell,q}}(\theta)&:=\left(\frac{1}{\ell}\sum_{k=1}^\ell\kmax_{1\leq i \leq N}|\langle\theta,e_i\rangle|^q\right)^{1/q}, \qquad \theta\in \SSS^{N-1}.
\end{align*}
Then, for all $1\leq \ell\leq N$ and every $\theta\in \SSS^{n-1}$,
\[
\E \,h_{K_{N,\ell,q}}(\theta)= c_n w\big(K_{\ell,q}\big),
\]
where $c_n=\frac{n\Gamma\left(1+\frac{n-1}{2}\right)}{\sqrt{2}\Gamma\left(1+\frac{n}{2}\right)}\approx\sqrt{n}$.
\end{lemma}

\begin{rmk}
Notice that if $\ell=N$, then $K_{\ell,q}=N^{-1/q}\B_{q^*}^N$.
\end{rmk}
\begin{proof}[of Lemma \ref{lem: support gaussian gives mean width}]
Let $\theta\in \SSS^{n-1}$. Since the random variables $g_i=\langle X_i,\theta\rangle$, $i\leq N$ are independent standard Gaussian random variables, we have
\begin{align*}
\E\, h_{K_{N,\ell,q}}(\theta)&=\E \left(\frac{1}{\ell}\sum_{k=1}^\ell \kmax_{1\leq i \leq N} |\langle X_i,\theta\rangle|^q\right)^{1/q} =\E \left(\frac{1}{\ell}\sum_{k=1}^\ell \kmax_{1\leq i \leq N} |g_i|^q\right)^{1/q}.
\end{align*}
Of course,
\[
\E \left(\frac{1}{\ell}\sum_{k=1}^\ell \kmax_{1\leq i \leq N} |g_i|^q\right)^{1/q} = \E \left(\frac{1}{\ell}\sum_{k=1}^\ell \kmax_{1\leq i \leq N} |\langle X_1,e_i \rangle|^q\right)^{1/q},
\]
and so integrating in polar coordinates, we obtain
\begin{align*}
\E\, h_{K_{N,\ell,q}}(\theta)
&=\frac{n |\B_2^n|}{(2\pi)^\frac{n}{2}}\int_0^\infty r^{n}e^{-\frac{r^2}{2}}\dint r\int_{\SSS^{n-1}}\left(\frac{1}{\ell}\sum_{k=1}^\ell \kmax_{1\leq i \leq N} |\langle \theta, e_i\rangle|^q\right)^{1/q}\dint\sigma_{n-1}(\theta)\cr
&=c_n w\big(K_{\ell,q}\big),
\end{align*}
where
\[
c_n := n |\B_2^n|\int_0^\infty r^{n}\frac{e^{-\frac{r^2}{2}}}{(2\pi)^\frac{n}{2}}\,\dint r = \frac{n\Gamma\left(1+\frac{n-1}{2}\right)}{\sqrt{2}\Gamma\left(1+\frac{n}{2}\right)} \approx \sqrt{n}.
\]
\end{proof}

The following lemma serves the purpose of estimating the quantity $w\big(K_{\ell,q}\big)$ from the previous lemma, which will be present in the other cases as well. It is a direct consequence of \cite[Example 16]{GLSW1}.
\begin{lemma}\label{EstimatesGaussianRandomVariablesq}
Let $N\in\N$ and $g_1,\dots g_N$ be independent standard Gaussian random variables. Then, for all $1\leq q\leq \log N$,
\[
\bigg(\frac{1}{\ell}\,\E \sum_{k=1}^\ell\kmax_{1\leq i \leq N}|g_i|^q\bigg)^{1/q}\approx\begin{cases}\sqrt{\log\frac{N}{\ell}}, & q\leq\log\frac{N}{\ell},\cr
\sqrt q ,& \log\frac{N}{\ell}\leq q\leq\log N.
\end{cases}
\]
\end{lemma}

\begin{proof}
By \cite[Example 16]{GLSW1}, if $1\leq q\leq \log N$, then
\[
c^{q/2}q^{q/2} \Vert(1)_{i=1}^N\Vert_{M_\ell} \leq \E \sum_{k=1}^\ell\kmax_{1\leq i \leq N}|g_i|^q \leq C^{q/2}q^{q/2}\Vert(1)_{i=1}^N\Vert_{M_\ell},
\]
where $c\in(0,1)$, $C\in(1,\infty)$ are absolute constants and
\[
M_\ell(t)=\begin{cases}
0, & t=0,\cr
\frac{1}{\ell}e^{-\frac{q}{(\ell t)^{2/q}}}, & t\in \Big(0,\frac{1}{\ell}\left(\frac{2q}{q+2}\right)^{q/2}\Big),\cr
\frac{(q+2)^{1+q/2}}{2^{q/2}q^{1+ q/2}}e^{-\frac{q}{2}}t-\frac{2}{eq\ell}e^{-\frac{q}{2}}, & t\geq\frac{1}{\ell}\left(\frac{2q}{q+2}\right)^{q/2}.
\end{cases}
\]
Let us compute $\Vert(1)_{i=1}^N\Vert_{M_\ell}=1/M_{\ell}^{-1}\left(\frac{1}{N}\right)$. Note that, since $q\geq 1$,
\[
 M_\ell\left(\frac{1}{\ell}\left(\frac{2q}{q+2}\right)^{q/2}\right)=\frac{e^{-\frac{q+2}{2}}}{\ell},
\]
which is greater than $1/N$ if and only if $q<2\log\frac{N}{\ell}-2$. In such case,
\[
M_{\ell}^{-1}\Big(\frac{1}{N}\Big)=\frac{q^{q/2}}{\ell\left(\log\frac{N}{\ell}\right)^{q/2}},
\]
and then
\[
\bigg(\frac{1}{\ell}\,\E\sum_{k=1}^\ell\kmax_{1\leq i \leq N}|g_i|^q\bigg)^{1/q}\approx\sqrt{\log\frac{N}{\ell}}.
\]
If $2\log\frac{N}{\ell}-2\leq q\leq\log N$, then
\[
M_{\ell}^{-1}\Big(\frac{1}{N}\Big)=\left(\frac{2q}{q+2}\right)^{q/2}\frac{1}{q+2}\left(\frac{2N+eq\ell e^{q/2}}{eN\ell}\right)
\]
and, consequently,
\[
\bigg(\E \sum_{k=1}^\ell\kmax_{1\leq i \leq N}|g_i|^q\bigg)^{1/q}  \approx \frac{N^{1/q}\ell^{1/q}\sqrt{q}}{N^{1/q}+\ell^{1/q}}
 = \frac{\sqrt{q}}{\frac{1}{\ell^{1/q}}+\frac{1}{N^{1/q}}}
 \approx \ell^{1/q}\sqrt{q}.
\]
Thus, since $q\leq\log N$,
\[
\bigg(\frac{1}{\ell}\,\E \sum_{k=1}^\ell\kmax_{1\leq i \leq N}|g_i|^q\bigg)^{1/q}  \approx \sqrt q.
\]
\end{proof}

As a consequence, we obtain the proof of Theorem \ref{TheoremMEanWidthGeneralq} in the Gaussian case.

\begin{proof}[Proof of Theorem \ref{TheoremMEanWidthGeneralq} -- Gaussian case]
Let $n,N\in\N$ with $N\geq n$ and let $X_1,\dots,X_N$ be independent standard Gaussian random vectors in $\R^n$.
Let also $1\leq \ell\leq N$,  $q\geq 1$ and  $\theta\in \SSS^{n-1}$. We are going to show that
\[
\E \,h_{K_{N,\ell,q}}(\theta)\simeq
\begin{cases}
\sqrt{\log(N/\ell)}, & q\leq\log(N/\ell),\cr
\sqrt q , & \log(N/\ell)\leq q\leq\log N,\cr
\sqrt{\log N},  & q\geq\log N.
\end{cases}
\]

Let us start with the upper bounds. Assume first that $q\geq\log N$. Then, in particular, $q\geq\log \ell$ and therefore
\[
\E\, h_{K_{N,\ell,q}}(\theta)\approx \E\, h_{K_{N,1,1}}(\theta)\approx \sqrt{\log N}.
\]
\vskip 1mm
If $q\leq \log N$, then Jensen's inequality implies
\[
\E \,h_{K_{N,\ell,q}}(\theta)\leq \left(\frac{1}{\ell}\,\E \sum_{k=1}^\ell \kmax_{1\leq i \leq N} |\langle X_i,\theta\rangle|^q\right)^{1/q}.
\]
Taking into account that for any $\theta\in \SSS^{n-1}$ we have that $\langle X_i,\theta\rangle$, $i\leq N$ are independent standard Gaussian random variables and using Lemma \ref{EstimatesGaussianRandomVariablesq}, we obtain the other two estimates.

Let us now prove the lower bounds. On the one hand, by Theorem \ref{EstimateGeneralq}, for any $q\geq 1$, we have
\[
\E\, h_{K_{N,\ell,q}}(\theta)\geq\E\max_{1\leq i\leq\lfloor N/\ell\rfloor}|\langle X_i,\theta\rangle|\approx\sqrt{\log(N/\ell)},
\]
which gives the right estimate if $1\leq q\leq\log(N/\ell)$. On the other hand, since $h_{K_{N,\ell,q}}(\theta)$ decreases in $\ell$ and because for any $\theta\in \SSS^{n-1}$ the random variables $\langle X_i,\theta\rangle$, $i\leq N$ are independent standard Gaussian random variables, we have that for any $\theta\in \SSS^{n-1}$
\[
\E\, h_{K_{N,\ell,q}}(\theta)\geq N^{-1/q}\,\E\bigg(\sum_{i=1}^N|\langle X_i,\theta\rangle|^q\bigg)^{1/q} = N^{-1/q}\,\E\,\|(g_i)_{i=1}^N\|_q,
\]
where $g_1,\dots,g_N$ are independent standard Gaussian random variables. Therefore, using the equivalence provided by \eqref{eq: p norm of gaussian random vector}, we obtain for any $\theta\in \SSS^{n-1}$ that
\[
\E\, h_{K_{N,\ell,q}}(\theta)\gtrsim \begin{cases}
\sqrt{q} , & q\leq\log N,\cr
\sqrt{\log N},  & q\geq\log N.
\end{cases}
\]
This obviously gives the right estimate whenever $q\geq\log(N/\ell)$.

Since the estimates hold for any $\theta\in S^{n-1}$, integrating on $\SSS^{n-1}$ with respect to $\sigma_{n-1}$ we obtain the result for the mean width
\end{proof}

\subsection{Isotropic log-concave random vectors -- the case $n\leq N \leq e^{\sqrt{n}}$}\label{subsec:isotropic}

In this subsection we consider random convex sets that arise by considering a $q^{th}$-moment of an average of order statistics of $1$-dimensional marginals of general independent log-concave random vectors in $\R^n$. Here we work in the regime $n\leq N \leq e^{\sqrt{n}}$ and the bounds are optimal (up to constants).

We start with a lemma that provides an upper bound on the mean width of our random convex sets in terms of an expression involving the $q^{th}$-moment of an average of order statistics of a sequence independent standard Gaussians.


\begin{lemma}\label{UpperBoundsInConvexBody}
Let $n,N\in \N$ with $n\leq N\leq e^{\sqrt n}$ and let $X_1,\dots X_N$ be independent isotropic log-concave random vectors in $\R^n$. Assume that $g_1,\dots, g_N$ are independent standard Gaussian random variables.
Then, for every $1\leq \ell\leq N$ and any $q\geq 1$,
\[
\E\, w(K_{N,\ell,q})\leq  C\left(\frac{1}{\ell}\,\E \sum_{k=1}^\ell\kmax_{1\leq i\leq N}|g_i|^q\right)^{1/q},
\]
where $C\in(0,\infty)$ is an absolute constant.
\end{lemma}

\begin{proof}
Let $G$ be a standard Gaussian random vector in $\R^n$. Then, by integration in polar coordinates,
\[
\E_X\E_G \,h_{K_{N,\ell,q}}(G)=c_n\E\, w(K_{N,\ell,q}),
\]
where $c_n=\frac{n\Gamma\left(1+\frac{n-1}{2}\right)}{\sqrt{2}\Gamma\left(1+\frac{n}{2}\right)}\approx\sqrt{n}$. On the other hand, notice that for every $1\leq i\leq N$, the random variables $\widetilde{g}_i=\langle X_i/\|X_i\|_2,G\rangle$ are non-independent standard Gaussian random variables for each realization of $X_1,\dots, X_N$. Take into account that, by Theorem 4 in \cite{GLSW1}, we have that for any sequence $\lambda_1,\dots,\lambda_N$ of real numbers
\[
\E_G\sum_{k=1}^\ell \kmax_{1\leq i\leq N}\lambda_i|\widetilde{g}_i|^q\leq\E_{G}\sum_{k=1}^\ell\kmax_{1\leq i\leq N}\lambda_i|g_i|^q,
\]
where $g_1,\dots,g_N$ are independent standard Gaussian random variables. Thus, taking into account that $\displaystyle{\kmax_{1\leq i\leq N}|\langle X_i,G_i\rangle|\leq\max_{1\leq i\leq N}\| X_i\|_2\cdot\kmax_{1\leq i\leq N}|\widetilde{g}_i|}$, we have
\begin{align*}
\E_X\E_G\, h_{K_{N,\ell,q}}(G)&\leq\E_X\left[\max_{1\leq i\leq N}\Vert X_i\Vert_2\left(\frac{1}{\ell}\,\E_G\sum_{k=1}^\ell \kmax\left|\widetilde{g}_i\right|^q\right)^{1/q}\right]\cr
&\leq\E_X\max_{1\leq i\leq N}\Vert X_i\Vert_2\cdot\left(\frac{1}{\ell}\,\E_G\sum_{k=1}^\ell \kmax\left|g_i\right|^q\right)^{1/q}.
\end{align*}
By Paouris' theorem (see Proposition \ref{thm:paouris}), since $N\leq e^{\sqrt n}$, we have that $\E_X\max_{1\leq i\leq N}\Vert X_i\Vert_2\lesssim\sqrt n$.
This proves the result. Notice that the second quantity is, up to absolute constants, of the order $\sqrt{n}w(K_{\ell,q})$, which we have estimated in the previous Section.
\end{proof}

Now, using the bounds for Gaussian random variables (Lemma \ref{EstimatesGaussianRandomVariablesq}), we can prove the estimates in Theorem \ref{TheoremMEanWidthGeneralq}.
\begin{proof}[Proof of Theorem \ref{TheoremMEanWidthGeneralq} -- General case]
Let $n,N\in \N$ with $n\leq N\leq e^{\sqrt n}$ and let $X_1,\dots X_N$ be independent random vectors in $\R^n$ with isotropic log-concave distribution $\mu$.
Let also $1\leq \ell\leq N$ and $q\geq 1$.

\vskip 1mm
We start with the upper bounds. If $q\geq\log N$, then $q\geq\log \ell$ and, for any $\theta\in \SSS^{n-1}$, $h_{K_{N,\ell,q}}(\theta)\approx h_{K_{N,1,1}}(\theta)$. Hence,
\[
\E\, w\big(K_{N,\ell,q}\big)\approx \E\, w(K_{N,1,1})\approx\sqrt{\log N}.
\]
If $q\leq\log N$, using Lemma \ref{UpperBoundsInConvexBody} and Lemma \ref{EstimatesGaussianRandomVariablesq}, we obtain the desired upper bounds.


\vskip 1mm
Let us now prove the lower estimates. If $q\geq\log N$, then, in particular, $q\geq\log \ell$ and hence, $\E\, h_{K_{N,\ell,q}}(\theta)\approx\E\, h_{K_{N,1,1}}(\theta)$. Therefore,
\[
\E\, w\big(K_{N,\ell,q}\big)\approx \E\, w\big(K_{N,1,1}\big)\approx\sqrt{\log N}.
\]
On the other hand, for all $\theta\in\SSS^{n-1}$ and for any $q\geq 1$, Theorem \ref{EstimateGeneralq} implies
\[
\E\, h_{K_{N,\ell,q}}(\theta)\geq\E\max_{1\leq i\leq\lfloor N/\ell\rfloor}|\langle X_i,\theta\rangle|. 
\]
Therefore,
\[
\E w(K_{N,\ell,q})\geq \E w(K_{\lfloor N/\ell\rfloor,1,1})\approx\sqrt{\log(N/\ell)},
\]
which is of the right order whenever $1\leq q\leq\log(N/\ell)$. On the other hand, since for any $\theta\in \SSS^{n-1}$ we know that $h_{K_{N,\ell,q}}(\theta)$ decreases in $\ell$ (see Lemma \ref{rmk:simple observations} (ii)), we obtain, for every $\theta\in \SSS^{n-1}$,
\[
h_{K_{N,\ell,q}}(\theta)\geq N^{-1/q}\left(\sum_{i=1}^N|\langle X_i,\theta\rangle|^q\right)^{1/q}.
\]
Now, if $q\leq \log N$, fix $m=\left\lfloor \frac{N}{e^q}\right\rfloor$ and take a partition $\sigma_1,\dots, \sigma_m$ of $\{1,\dots, N\}$ such that $|\sigma_j|\geq e^q$ for every $1\leq j\leq m$. We then obtain
\begin{align*}
h_{K_{N,\ell,q}}(\theta)&\geq N^{-1/q}\bigg(\sum_{j=1}^k\sum_{i\in\sigma_j}|\langle X_i,\theta\rangle|^q\bigg)^{1/q}\cr
&\geq N^{-1/q}\bigg(\sum_{j=1}^k\Big(\max_{i\in\sigma_j}|\langle X_i,\theta\rangle|\Big)^q\,\bigg)^{1/q}.
\end{align*}
Now, by Jensen's inequality,
\begin{align*}
\E\, h_{K_{N,\ell,q}}(\theta)
&\geq N^{-1/q}\bigg(\sum_{j=1}^k\left(\E\max_{i\in\sigma_j}|\langle X_i,\theta\rangle|\bigg)^q\,\right)^{1/q}\cr
&\geq N^{-1/q}\inf_{1\leq j \leq k}\E\max_{i\in\sigma_j}|\langle X_i,\theta\rangle|k^{1/q}\cr
&\geq \inf_{1\leq j \leq k}\E\max_{i\in\sigma_j}|\langle X_i,\theta\rangle|.
\end{align*}

Note that for any $\sigma_j\in\{\sigma_1,\dots,\sigma_k\}$, Markov's inequality implies that, for any $\alpha\geq 0$,
\begin{align*}
\E\max_{i\in\sigma_j}|\langle X_i,\theta\rangle|&\geq\alpha\,\Pro\big(\max_{i\in\sigma_j}|\langle X_i,\theta\rangle|\geq\alpha\big)\cr
&=\alpha\Big[1-\Pro\big(|\langle X_1,\theta\rangle|<\alpha\big)^{|\sigma_j|}\Big].
\end{align*}
Choosing $\alpha=h_{K_{\frac{1}{|\sigma_j|}}}(\theta)$, we have that $\Pro(|\langle X_1,\theta\rangle|<\alpha)=1-\frac{1}{|\sigma_j|}$. Therefore,
\[
\E\max_{i\in\sigma_j}|\langle X_i,\theta\rangle|\geq c h_{K_{\frac{1}{|\sigma_j|}}}(\theta).
\]
By Lemma \ref{rem:relation floating and Lq centroid bodies} there exist absolute constants $c_1,c_2\in(0,\infty)$ such that, for any $\delta\in(0,1/e)$,
\[
c_1Z_{\log\frac{1}{\delta}}(\mu)\subseteq K_\delta\subseteq c_1Z_{\log\frac{1}{\delta}}(\mu).
\]
Therefore, if we have $q\leq \log N$, then, for any $\theta\in S^{n-1}$,
\[
\E \,h_{K_{N,\ell,q}}(\theta)\geq c\inf_{1\leq j \leq k} h_{Z_{\log |\sigma_j|}(\mu)}(\theta)\approx h_{Z_q(\mu)}(\theta).
\]
Thus, if $q\leq\log N$,
\begin{align*}
\E \,w\big(K_{N,\ell,q}\big) & = \int_{\SSS^{n-1}}h_{K_{N,\ell,q}}(\theta)\,\dint\sigma_{n-1}(\theta) \cr
& \geq\int_{\SSS^{n-1}}h_{Z_q(K)}(\theta)\,\dint\sigma_{n-1}(\theta)=w\big(Z_q(\mu)\big).
\end{align*}
Thus, if $q\leq\log N$ and $N\leq e^{\sqrt{n}}$,
\[
\E\, w\big(K_{N,\ell,q}\big)\geq c\sqrt{q}.
\]
\end{proof}

\begin{rmk}\label{RemarkIsotroic}
Note that the last proof and an application of Jensen's inequality show that, for any $q\geq 1$ and any $\theta\in \SSS^{n-1}$,
\[
h_{Z_q(K)}(\theta)\approx \E\, h_{K_{N,N,q}}(\theta),
\]
as long as $N\geq e^{q}$ with no upper bound on the number of points we can take or on the parameter $q$.
Thus, for any $q\geq 1$ and $N\geq e^q$,
\[
w\big(Z_q(K)\big)\approx \E\, w\big(K_{N,N,q}\big).
\]
If $N\leq e^q$, $\E\, h_{K_{N,N,q}}(\theta)\approx \E\, h_{K_{N,1,1}}(\theta)$ and therefore,
\[
\E \,w\big(K_{N,N,q}\big)\approx \E\, w(K_{N,1,1}).
\]
\end{rmk}


\subsection{Isotropic log-concave random vectors -- the case $e^{\sqrt n}\leq N\leq e^n$}\label{subsec:isotropicManyPoints}

In this subsection we consider again the isotropic log-concave random model. This time we work in the regime $e^{\sqrt n}\leq N\leq e^n$ and the following estimates can be obtained:

\begin{thm}\label{TheoremMEanWidthGeneralUpperBoundManyPoints}
	Let $n,N\in\N$ with $e^{\sqrt{n}}\leq N\leq e^{n}$ and let $X_1,\dots,X_N$ be independent random vectors in $\R^n$ distributed according to an isotropic log-concave probability law $\mu$ on $\R^n$.
	Then, for all $1\leq \ell\leq N$ and any $q\geq 1$,
	\[
	\E\, w\big(K_{N,\ell,q}\big)\lesssim\begin{cases}\frac{\log N}{\sqrt n}\sqrt{\log(N/\ell)}, & q\leq\log(N/\ell),\cr
\frac{\log N}{\sqrt n}\sqrt q , & \log(N/\ell)\leq q\leq\log N,\cr
\sqrt{\log N}(\log\log N)^2,  & q\geq\log N.\end{cases}
	\]
\end{thm}
\begin{proof}
If $q\geq\log N$, then $q\geq\log \ell$ and, for any $\theta\in \SSS^{n-1}$, $h_{K_{N,\ell,q}}(\theta)\approx h_{K_{N,1,1}}(\theta)$. Hence, using Theorem 1.3 in \cite{GHT}, we have
\[
\E\, w\big(K_{N,\ell,q}\big)\approx \E\, w(K_{N,1,1})\lesssim \sqrt{\log N}(\log\log N)^2.
\]
If $q\leq\log N$ we have, like in the proof of Lemma \ref{UpperBoundsInConvexBody}, that
$$
\E\, w\big(K_{N,\ell,q}\big)\lesssim\frac{1}{\sqrt n}\,\E_X\max_{1\leq i\leq N}\Vert X_i\Vert_2\cdot\left(\frac{1}{\ell}\,\E_G\sum_{k=1}^\ell \kmax\left|g_i\right|^q\right)^{1/q}.
$$
Using Paouris' theorem (Proposition \ref{thm:paouris}) and Lemma \ref{EstimatesGaussianRandomVariablesq}, we obtain the desired upper bounds.
\end{proof}

Regarding lower bounds, the same proofs give the following estimates:
\begin{thm}\label{TheoremMEanWidthGeneralLowerBoundManyPoints}
	Let $n,N\in\N$ with $e^{\sqrt{n}}\leq N\leq e^{n}$ and let $X_1,\dots,X_N$ be independent random vectors in $\R^n$ distributed according to an isotropic log-concave probability law $\mu$ on $\R^n$.
	Then, for all $1\leq \ell\leq N$ and any $q\geq 1$,
	\[
	\E\, w\big(K_{N,\ell,q}\big)\gtrsim\begin{cases}\max\big\{w(Z_{\log(1+N/\ell)}(\mu)),w(Z_q(\mu))\big\} , & 1\leq q\leq\log N,\cr
w(Z_{\log N}(\mu)),  & q\geq\log N.\end{cases}
	\]
\end{thm}

\subsection{Random vectors uniformly distributed on $\ell_p^n$-spheres}\label{subsec:spheres}

In this part we consider random convex sets which arise from considering the $q^{th}$ moment of an average of order statistics of the $1$-dimensional marginals of independent random points which are chosen with respect to the cone measure $\bM_{\B_p^n}$ on the sphere $\SSS_p^{n-1}$ of $\ell_p^n$.


We will now present the proof of Theorem \ref{TheoremBpnq}. For $p\geq 2$, the upper bounds follow from standard norm estimates, while in the case $1\leq p <2$, we need the large deviation estimate for the cone measure due to G. Schechtman and J. Zinn (see Proposition \ref{thm:deviation cone measure}). To obtain the corresponding lower bounds, we use the coupling argument that was recently used in \cite{HPT2016} to reduce the case of $\ell_p^n$-spheres to the isotropic case, that is, $\B_p^n/|\B_p^n|^{1/n}$.

\begin{proof}[Proof of Theorem \ref{TheoremBpnq} -- $\SSS_p^{n-1}$ case]
We first present the proof of the upper bounds.
\vskip 1mm
\noindent{\em Upper bounds:}
We have, as in the proof of Lemma \ref{UpperBoundsInConvexBody},
\[
\E\, w\big(K_{N,\ell,q}\big)\lesssim\frac{1}{\sqrt n}\,\E_X\max_{1\leq i\leq N}\Vert X_i\Vert_2\cdot\left(\frac{1}{\ell}\,\E_G\sum_{k=1}^\ell \kmax\left|g_i\right|^q\right)^{1/q}.
\]

We now consider two different cases.
\vskip 1mm
Let $2\leq p < \infty$. Then, since $R(\B_p^n)=n^{\frac{1}{2}-\frac{1}{p}}$, we have
\[
\E\, w\big(K_{N,\ell,q}\big)\leq n^{-\frac{1}{p}}\left(\frac{1}{\ell}\,\E_{G}\sum_{k=1}^\ell\kmax_{1\leq i\leq N}|g_i|^q\right)^{1/q}.
\]
\vskip 1mm
Let $1\leq p < 2$ and
$T$ denote the constant that appears in Proposition \ref{thm:deviation cone measure} (we may assume $T\geq 1$). Applying this proposition with $q=2$ there, we obtain
\begin{align*}
\E_X\max_{1\leq j \leq N} \|X_j\|_2 & = \int_0^\infty \Pro\Big(\max_{1\leq j \leq N} \|X_j\|_2 \geq u\Big) \,\dint u \cr
& = n^{1/2-1/p}\int_0^T \Pro\bigg(\max_{1\leq j \leq N} \|X_j\|_2 \geq \frac{t}{n^{1/p-1/2}}\bigg) \,\dint t \cr
& \quad + n^{1/2-1/p}\int_T^\infty \Pro\bigg(\max_{1\leq j \leq N} \|X_j\|_2 \geq \frac{t}{n^{1/p-1/2}}\bigg) \,\dint t \cr
& \leq Tn^{1/2-1/p} + n^{1/2-1/p}N\int_T^\infty \Pro\bigg( \|X_1\|_2 \geq \frac{t}{n^{1/p-1/2}}\bigg) \,\dint t \cr
& \leq Tn^{1/2-1/p} + n^{1/2-1/p}N\int_T^\infty \exp\Big(-\frac{t^pn^{p/2}}{c}\Big) \,\dint t \cr
& \leq Tn^{1/2-1/p} + n^{1/2-1/p} \int_T^\infty\exp\Big(-\frac{t^pn^{p/2}}{c_1}\Big) \,\dint t,
\end{align*}
where $c_1\in(0,\infty)$ is an absolute constant. In the last step we used that $N\leq e^{c_2 n^{p/2}}$ for a suitably small enough constant $c_2\in(0,\infty)$. Using a change of variable in the second term (and taking the integral from 0 to $\infty$), we obtain
\[
\E_X\max_{1\leq j \leq N} \|X_j\|_2 \leq Tn^{1/2-1/p} + \frac{c_1^{1/p}}{p}n^{-1/p}\Gamma(1/p).
\]
Taking into account that for $p\in[1,2)$ the constant $T\in(0,\infty)$ is absolute, we obtain
\[
\E_X\max_{1\leq j \leq N} \|X_j\|_2 \lesssim C_1n^{1/2-1/p}.
\]
Let us conclude the proof. As a consequence of Lemma \ref{EstimatesGaussianRandomVariablesq}, we obtain the result whenever $1\leq q\leq\log N$. If $q\geq\log N$, then $q\geq\log \ell$, and hence, for any $\theta\in \SSS^{n-1}$, $h_{K_{N,\ell,q}}(\theta)\approx h_{K_{N,1,1}}(\theta)$. Thus, since $1\leq \log N$,
\[
\E\, w\big(K_{N,\ell,q}\big)\approx \E\, w(K_{N,1,1})\lesssim n^{-1/p}\sqrt{\log N}.
\]
\vskip 1mm
\noindent{\em Lower bounds:} Let $Y_1,\dots, Y_N$ be independent random vectors uniformly distributed in $\B_p^n$ and let $\widetilde{K}_{N,\ell,q}$ be the random convex body defined by
\[
h_{\widetilde{K}_{N,\ell,q}}(\theta) :=\left(\frac{1}{\ell}\sum_{k=1}^\ell \kmax_{1\leq i \leq N} |\langle Y_i,\theta\rangle|^q\right)^{1/q}, \qquad \theta\in \SSS^{n-1}.
\]
Since $\B_p^n/|\B_p^n|^{1/n}$ is isotropic and its isotropic constant is bounded, Theorem \ref{TheoremMEanWidthGeneralq} implies that
\[
\E\, w\big(\widetilde{K}_{N,\ell,q}\big)\approx n^{-\frac{1}{p}}\min\left\{\max\left\{\sqrt q, \sqrt{\log(N/\ell)}\,\right\},\sqrt{\log N}\right\}.
\]
Taking into account that, by the definition of the cone measure, the random vector $X_i$ has the same distribution as $Y_i/\Vert Y_i\Vert_p$ for each $i\leq N$, and that, since $Y_i\in \B_p^n$, $\widetilde{K}_{N,\ell,q}\subseteq L_{N,\ell,q}$, where $ L_{N,\ell,q}$ is the random convex body defined by
$$
h_{L_{N,\ell,q}}(\theta) :=\left(\frac{1}{\ell}\sum_{k=1}^\ell \kmax_{1\leq i \leq N} \Big|\Big\langle \frac{Y_i}{\Vert Y_i\Vert_p},\theta\Big\rangle\Big|^q\right)^{1/q}\,, \qquad \theta\in \SSS^{n-1},
$$
we have that
\[
\E\, w\big(K_{N,\ell,q}\big)\geq \E\, w\big(\widetilde{K}_{N,\ell,q}\big)\gtrsim
\begin{cases}
n^{-1/p}\sqrt{\log(N/\ell)}, & 1\leq q\leq\log(N/\ell),\cr
n^{-1/p}\sqrt q, & \log(N/\ell)\leq q\leq\log N,\cr
n^{-1/p}\sqrt{\log N}, & q\geq \log N.
\end{cases}
\]
\end{proof}

\begin{rmk}
Taking a look at the proof reveals that similar results can be obtained when the random points are chosen with respect to the cone probability measure $\bM_K$ from the boundary $\bd\, K$ of an isotropic convex body $K$ for which $\E\max_{1\leq i\leq N}\Vert X_i\Vert_2\leq \sqrt n\,|K|^{1/n}$ holds.
\end{rmk}

\vskip 1mm

\noindent\textbf{Acknowledgement}
We would like to thank the anonymous referees for a careful reading of the manuscript and helpful comments that improved the presentation of this work. The first named author is partially supported by MICINN MTM2013-42105, MTM2016-77710-P, and Bancaja P1-1B2014-35 projects.



\bibliographystyle{plain}
\bibliography{random_convex_sets}



\end{document}